\documentclass[12pt,a4paper]{amsproc}

\DeclareMathOperator{\RE}{Re}

\numberwithin{equation}{section}
\newtheorem{theorem}{Theorem}[section]
\newtheorem{lemma}[theorem]{Lemma}

\makeatletter
\newcommand{\subjclassname@later}{\textup{2010} Mathematics Subject Classification}
\makeatother

\begin{document}
\title[Radii of Starlikeness ]{Radii of Starlikeness Associated with the Lemniscate  of Bernoulli and the Left-Half Plane}

\author[R. M. Ali]{Rosihan M. Ali}
\address{School of Mathematical Sciences, Universiti Sains
Malaysia, 11800 USM Penang, Malaysia} \email{rosihan@cs.usm.my}

\author[N. Jain]{Naveen Jain}
\address{Department of Mathematics, University of Delhi,
Delhi-110007, India}
\email{naveenjain05@gmail.com}

\author{V. Ravichandran}
\address{Department of Mathematics, University of Delhi,
Delhi-110007, India} \email{vravi@maths.du.ac.in}

\subjclass[later]{Primary 30C80; Secondary 30C45} 

\thanks{The work presented here was supported in part by grants from Universiti Sains Malaysia and University of Delhi.}

\begin{abstract} A normalized analytic function $f$  defined on the open unit disk in the complex
plane is  in the class $\mathcal{SL}$  if $zf'(z)/f(z)$ lies in the region bounded by the right-half of the lemniscate of Bernoulli given by
$|w^2-1|<1$. In the present investigation, the $\mathcal{SL}$-radii for certain well-known classes of functions are obtained. Radius problems
associated with the left-half plane are also investigated for these classes.
\end{abstract}

\maketitle

\section{Introduction}Let $\mathcal{A}_n$ denote the class of analytic functions  in the unit disk
$\mathbb{D}:= \{z:\  |z| <1\}$ of the form $f(z)=z+\sum_{k=n+1}a_k z^k$, and let $\mathcal{A}:=\mathcal{A}_1$. Let $\mathcal{S}$ denote the
subclass of $\mathcal{A}$ consisting of  univalent functions. Let $ \mathcal{SL}$ be the class of functions defined by
\[
\mathcal{SL}:=\left\{f\in \mathcal{A}:\left|\left(\frac{zf'(z)}{f(z)}\right)^2-1\right|<1\right\}\quad (z\in \mathbb{D}).
\]
Thus a function $f\in \mathcal{SL}$  if $zf'(z)/f(z)$ lies in the region bounded by the right-half of the lemniscate of Bernoulli given by
$|w^2-1|<1$. For two functions $f$ and $g$  analytic in $\mathbb{D}$, the function $f$ is said to be \emph{subordinate} to $g$, written
$f(z)\prec g(z) \quad (z\in \mathbb{D})$, if there exists a function $w$ analytic in $\mathbb{D}$ with $ w(0) = 0$ and $|w(z)|<1$ such that
$f(z)=g(w(z)).$ In particular, if the function $g$ is univalent in $\mathbb{D}$, then $f(z)\prec g(z)$ is equivalent to $f(0) =g(0)$ and
$f(\mathbb{D})\subset g(\mathbb{D})$. In terms of subordination, the class $\mathcal{SL}$ consists of normalized analytic functions $f$
satisfying $zf'(z)/f(z) \prec \sqrt{1+z}$.  This class $ \mathcal{SL}$  was introduced by  Sok\'o\l\ and Stankiewicz \cite{sokol96}.  Paprocki\
and Sok\'o\l \cite{sokol96b} discussed a more general class $\mathcal{S}^*(a,b)$ consisting of normalized analytic  functions $f$ satisfying
$\left|[zf'(z)/f(z)]^a-b\right|<b$, $b\geq\frac12$, $a\geq 1$.

Recall that a  function $f\in \mathcal{A}$ is starlike if $f(\mathbb{D})$ is starlike with respect to $0$.    Similarly, a function $f\in
\mathcal{A}$  is convex if $f(\mathbb{D})$ is convex. Analytically, a function $f\in\mathcal{A}$ is starlike or convex if the  following
respective subordinations hold:
\[ \frac{zf'(z)}{f(z)}\prec \frac{1+z}{1-z}, \quad \text{ or}   \quad 1+ \frac{zf''(z)}{f'(z)}\prec  \frac{1+z}{1-z}  . \]

Ma and Minda  \cite{mamin2} gave a unified presentation of various subclasses of starlike and convex functions by replacing the superordinate
function $(1+z)/(1-z)$ by a more  general function $\varphi$.     They considered  analytic functions  $\varphi$  with   positive real part that
map the unit disk $\mathbb{D}$ onto regions starlike with respect to 1, symmetric with respect to the real axis and normalized by
$\varphi(0)=1$. They  introduced the following classes that include several well-known classes as special cases:
\[
 \mathcal{ST} (\varphi) := \left\{f\in \mathcal{A} \left |\  \frac{zf'(z)}{f(z)}\prec \varphi(z)\right \}\right.
\]
and
\[
\mathcal{CV}(\varphi) := \left\{f\in \mathcal{A}\left|\ 1+ \frac{zf''(z)}{f'(z)}\prec \varphi(z)\right \}\right..
\] For $0 \leq \alpha<1$,
\[
\mathcal{ST}(\alpha):=\mathcal{ST}((1+(1-2\alpha)z)/(1-z)),\quad
 \mathcal{CV}(\alpha):=\mathcal{CV}((1+(1-2\alpha)z)/(1-z))
 \]
 are  the subclasses of $\mathcal{S}$ consisting of starlike and convex  functions
of order $\alpha$ in $\mathbb{D}$ respectively.  Then $\mathcal{ST}:=\mathcal{ST}(0)$,
$\mathcal{CV}:=\mathcal{CV}(0)$ are the well-known classes of starlike and convex  functions
  respectively. Also let
  \[ \mathcal{ST}_n(\alpha) := \mathcal{A}_n \cap \mathcal{ST}(\alpha), \quad \mathcal{CV}_n(\alpha): = \mathcal{A}_n \cap \mathcal{CV}(\alpha), \quad
  \mathcal{SL}_n := \mathcal{A}_n \cap \mathcal{SL}. \]

\noindent Since $\mathcal{SL}=\mathcal{ST}(\sqrt{1+z})$, distortion, growth, and rotation results for the class $\mathcal{SL}$ can conveniently
be obtained by applying the corresponding results in \cite{mamin2}.

The radius of a property $P$ in a set of functions $\mathcal{M}$, denoted by $R_P(\mathcal{M})$, is the largest number $R$ such that every
function in the set $\mathcal{M}$ has the property $P$ in each disk $\mathbb{D}_r=\{z\in\mathbb{D} : |z|<r \}$ for every $r<R$. For example, the
radius of convexity in the class $\mathcal{S}$ is $2-\sqrt{3}$. Sok\'o\l\ and Stankiewicz \cite{sokol96} determined the radius of convexity for
functions in the class $\mathcal{SL}$.  They have also obtained structural formula, growth and distortion theorems for these functions.
Estimates for the first few coefficients of functions in this class can be found in \cite{sokol09}.  Recently,   Sok\'o\l \ \cite{sokol09b}
determined
 various  radii for functions belonging to the class $\mathcal{SL}$; these include the radii of convexity,
 starlikeness and   strong starlikeness of order $\alpha$. In contrast, in our present investigation, we
 compute the $\mathcal{SL}$-radius for functions belonging to several interesting classes. Unlike the radii problems associated
 with starlikeness and convexity, where a central feature is the estimates for the real part of the expressions $zf'(z)/f(z)$ or $1+zf''(z)/f'(z)$
 respectively, the $\mathcal{SL}$-radius problems for classes of functions are tackled by first finding the disk
 that contains the values of $zf'(z)/f(z)$ or $1+zf''(z)/f'(z)$. This technical result will be presented in the next section.

Another interesting class is  $ \mathcal{M}(\beta)$, $\beta < 1$, defined by
\[ \mathcal{M}(\beta):=\left\{ f\in\mathcal{A}:
\RE\left( \frac{zf'(z)}{f(z)}  \right) < \beta,\quad z\in\Delta\right\}. \]   The class $ \mathcal{M}(\beta)$ was investigated by Uralegaddi
{\em et al.} \cite{ural}, while its subclass was investigated by Owa and Srivastava \cite{owa2}. We let $\mathcal{M}_n(\beta) := \mathcal{A}_n
\cap \mathcal{M}(\beta)$.
In the present paper, radius problems related to $ \mathcal{M}(\beta)$ will also be investigated. Related radius problem for this class can be found in \cite{rmavr} and \cite{vrkgs}. The following
definitions and results will be required.

 An analytic  function $p(z)=1+c_nz^n+\cdots$  is a function with positive real part if
$\RE p(z)>0$. The class of all such functions is denoted by $\mathcal{P}_n$. We also denote the subclass of $\mathcal{P}_n$ satisfying $\RE
p(z)>\alpha$, $0\leq \alpha <1$, by $\mathcal{P}_n(\alpha)$.  More generally, for $-1\leq B <A \leq 1$, the class $\mathcal{P}_n[A,B]$ consists
of  functions $p$ of the form $p(z)=1+c_nz^n+\cdots$ satisfying
\[p(z) \prec \frac{1+Az}{1+Bz}. \]

\begin{lemma} \cite{mac1}\label{lem2} If $p\in \mathcal{P}_n$, then
\[ \left|\frac{zp'(z)}{p(z)}\right| \leq \frac{2nr^{n}}{1-r^{2n}}  \quad (|z|=r<1). \]
\end{lemma}

\begin{lemma} \cite{vraviron} \label{lem3} If $p\in P_n[A,B]$, then
 \[\left|p(z)-\frac{1-ABr^{2n}}{1-B^2r^{2n}}\right|\leq\frac{(A-B)r^n}{1-B^2r^{2n}}\quad (|z| = r  < 1).\]
In particular, if $p\in P_n(\alpha)$, then \[
 \left|p(z) - \frac{1+(1-2\alpha)r^{2n}}{1-r^{2n}} \right| \leq \frac{2(1-\alpha)r^n}{1-r^{2n}}
\quad (|z| = r  < 1).\]
\end{lemma}
\section{The $\mathcal{SL}_n$-Radius Problems}

In this section, three special classes of functions will be considered. First is the class
\[ \mathcal{S}_n := \left\{f\in\mathcal{A}_n: \frac{f(z)}{z} \in\mathcal{P}_n \right\}. \] For this class, we shall find its $\mathcal{SL}_n$-radius, denoted by
$R_{\mathcal{SL}_n}(\mathcal{S}_n)$.

\begin{theorem} \label{th3}
The $\mathcal{SL}_n$-radius for the class $\mathcal{S}_n$ is
\[
R_{\mathcal{SL}_n} (\mathcal{S}_n)=\left\{\frac{\sqrt{2}-1}{n+\sqrt{n^2+(\sqrt{2}-1)^2}}\right\}^{1/n}.
\]
This radius is sharp.
\end{theorem}

\begin{proof} Define the function $h$ by
 \[ h(z)=\frac{f(z)}{z}. \]  Then the function $h\in\mathcal{P}_n$ and
  \[ \frac{zf'(z)}{f(z)}-1 = \frac{zh'(z)}{h(z)}. \]
Applying Lemma~\ref{lem2} to the function $h$ yields \[\left|\frac{zf'(z)}{f(z)}-1\right|\leq\frac{2nr^n}{1-r^{2n}}. \] Notice that if
$|w-1|<\sqrt{2}-1$, then $|w+1|\leq \sqrt{2}+1$ and hence $|w^2-1|\leq 1$. Thus the above disk lies inside the lemniscate $|w^2-1|<1$ if
\[ \frac{2nr^{n}}{1-r^{2n}}\leq \sqrt2-1. \] Solving this inequality for $r$ yields the desired $\mathcal{SL}_n$-radius for the class $\mathcal{S}_n$.

Now consider the function $f$ defined by
\[ f(z) = \frac{z+z^{n+1}}{1-z^n}. \] Clearly the function $f$ satisfies the hypothesis of the theorem and
\[ \frac{zf'(z)}{f(z)} = 1+ \frac{2nz^n}{1-z^{2n}}.\]
At $z=R$ where $R$ is the $\mathcal{SL}_n$-radius for the class $\mathcal{S}_n$  given in the theorem, routine computations show that
 \[\left|\left(\frac{zf'(z)}{f(z)}\right)^2-1\right|= 1. \]  This proves that the result is sharp.
\end{proof}

The following technical lemma will be useful in our subsequent investigations.

\begin{lemma}\label{lem25}  For $0 < a< \sqrt{2}$, let $r_a$ be given by
\[ r_a= \begin{cases} \left(\sqrt{1-a^2}-(1-a^2)\right)^{1/2} &\quad ( 0 <a \leq 2 \sqrt{2}/3) \\
\sqrt{2}-a & \quad (2 \sqrt{2}/3 \leq a <\sqrt{2}),
\end{cases} \] and for $a>0$, let $R_a$ be given by
\[ R_a= \begin{cases}\sqrt{2}-a &\quad ( 0 <a \leq 1/ \sqrt{2}) \\
a & \quad (1/\sqrt{2} \leq a ).
\end{cases} \]
 Then  \[ \{ w: |w-a|< r_a\} \subseteq \{w: |w^2-1|<1\}  \subseteq \{ w: |w-a|< R_a\} . \]
\end{lemma}

\begin{proof} The equation of the lemniscate of Bernoulli is
  \[(x^2+y^2)^2-2(x^2-y^2)=0\] and the  parametric equations of its right-half is given by
\[x(t)=\frac{\sqrt{2}\cos t}{1+\sin^2t},\qquad y(t)=\frac{\sqrt{2}\sin t\cos t}{1+\sin^2t}, \quad \left(-\frac{\pi}{2} \leq t\leq \frac{\pi}{2} \right).
\]
The square of the distance from the point $(a,0)$ to the points on the lemniscate is given by
\begin{align*}
z (t)& =(a-x(t))^2+(y(t))^2\\
& =a^2+\frac{2(\cos^2t-\sqrt{2}a\cos t)}{1+\sin^2t},
\end{align*}
and its derivative is
\[ z'(t) =2\frac{(-4\cos t+\sqrt{2}a(2+\cos^2t))\sin t}{(1+\sin^2t)^2}.\]
Clearly $z'(t)=0$ if and only if
\[t=0 \quad \text{ or } \quad \cos t=\frac{\sqrt{2}(1\pm \sqrt{1-a^2})}{a}. \] Note that for $a>1$, the numbers
${\sqrt{2}(1\pm\sqrt{1-a^2})}/{a}$ are complex and for $0<a\leq 1$, the number ${\sqrt{2}(1+\sqrt{1-a^2})}/{a}>1$.
For $0 < a <1$, the number ${\sqrt{2}(1-\sqrt{1-a^2})}/{a}$ lies between -1 and 1 if and only if $0 < a \leq 2\sqrt{2}/3$.

Let  us first assume that  $0<a\leq2\sqrt{2}/{3}$ and $t=t_0$ be given by \[\cos t_0=\frac{\sqrt{2}(1-\sqrt{1-a^2})}{a}.\]  Since
  \[\min\{z(\pi/2), z(-\pi/2), z(0), z(t_0)\} = z(t_0),\]
  it follows that $\min \sqrt{z(t)} = \sqrt{z(t_0)}$ .  A calculation shows that
\[  z(t_0) =\sqrt{1-a^2}-(1-a^2).\] Hence
\[  r_a= \min \sqrt{z(t)} = \sqrt{\sqrt{1-a^2}-(1-a^2)}.   \]

Let us next assume that $ {2\sqrt{2}}/{3}\leq a<\sqrt{2}$. In this case, \[\min\{z(\pi/2), z(-\pi/2), z(0)\} = z(0),\]  and thus $z(t)$ attains
its minimum value at $t=0$ and
\[ r_a= \min \sqrt{z(t)}  =\sqrt{2}-a .\]

Now consider  $0<a\leq1/\sqrt{2}$ and $t=t_0$ be given by \[\cos t_0=\frac{\sqrt{2}(1-\sqrt{1-a^2})}{a}.\]  It is easy to see that
\[ \max\{z(\pi/2), z(-\pi/2), z(0), z(t_0)\} = z(0), \] and thus
\[  R_a= \max \sqrt{z(t)} = \sqrt{2}-a.   \] Similarly, for $a\geq 1/\sqrt{2}$,
\[\max\{z(\pi/2), z(-\pi/2), z(0)\} = z(\pi/2),\]
and hence
\[  R_a= \max \sqrt{z(t)} =  a.  \qedhere \]
\end{proof}

Now consider the subclass  $\mathcal{CS}_n(\alpha)$  consisting of close-to-starlike functions of type  $\alpha$ defined by
\[ \mathcal{CS}_n(\alpha):=\left\{  f\in\mathcal{A}_n:  \frac{f}{g} \in \mathcal{P}_n, \quad g\in \mathcal{ST}_n(\alpha)  \right\}.  \]
The $\mathcal{SL}_n$-radius for this class is given in the following theorem.

\begin{theorem}\label{th1}  The $\mathcal{SL}_n$-radius for the class $\mathcal{CS}_n(\alpha)$ is given by
\[R_{\mathcal{SL}_n} (\mathcal{CS}_n(\alpha))=\left( \frac{\sqrt{2}-1}{(1+n-\alpha)+\sqrt{(1+n-\alpha)^2+(1-2\alpha+\sqrt2)(\sqrt2-1)}}\right)^{1/n}\]
This radius is sharp.
\end{theorem}
\begin{proof}Let $g$ be a starlike function of order $\alpha$ with
$h(z)= {f(z)}/{g(z)}\in \mathcal{P}_n$. Then $zg'(z)/g(z)$ is in $\mathcal{P}_n(\alpha)$ and from Lemma~\ref{lem3},
\begin{equation} \label{eq1} \left|\frac{zg'(z)}{g(z)}-\frac{1+(1-2\alpha)r^{2n}}{1-r^{2n}}\right|\leq\frac{2(1-\alpha)r^n}{1-r^{2n}}.\end{equation}
Applying Lemma~\ref{lem2} yields
\begin{equation} \label{eq2} \left|\frac{zh'(z)}{h(z)}\right|\leq\frac{2nr^n}{1-r^{2n}}.\end{equation}
Now
\begin{equation} \label{eq3}
\frac{zf'(z)}{f(z)} = \frac{zg'(z)}{g(z)}+\frac{zh'(z)}{h(z)},
\end{equation}
and using \eqref{eq1}--\eqref{eq3}, it follows that \begin{equation} \label{eq4}
\left|\frac{zf'(z)}{f(z)}-\frac{1+(1-2\alpha)r^{2n}}{1-r^{2n}}\right|\leq\frac{2(1+n-\alpha)r^n}{1-r^{2n}}.\end{equation} Since the center of
the disk in \eqref{eq4} is greater than 1, from Lemma ~\ref{lem25}, it is seen that the points $w$ are inside the lemniscate $|w^2-1|<1$ if
\[\frac{2(1+n-\alpha)r^n}{1-r^{2n}}\leq\sqrt2-\frac{1+(1-2\alpha)r^{2n}}{1-r^{2n}}. \] The last inequality reduces to $
(1-2\alpha+\sqrt{2})r^{2n} + 2(1+n-\alpha)r^n -(\sqrt{2}-1)\leq0$. Solving this latter inequality results in the value of $R=R_{\mathcal{SL}_n}
(\mathcal{CS}_n(\alpha))$.

The function $f$ given by
\[ f(z) = \frac{z(1+z^n)}{(1-z^n)^{(n+2-2\alpha)/n}} \]
satisfies the hypothesis of Theorem ~\ref{th1} with $g(z)=z/(1-z^n)^{(2-2\alpha)/n}$. It is easy to see that, for $z=R=R_{\mathcal{SL}_n}
(\mathcal{CS}_n(\alpha))$,
\[ \left|\left(\frac{zf'(z)}{f(z)}\right)^2-1\right| =\left| \frac{[1+(1-2\alpha)R^{2n}+2(1+n-\alpha)R^n]^2}{(1-R^{2n})^2}-1 \right| =1 .\]
This shows that the result is sharp.
\end{proof}

For $-1\leq B < A \leq 1$, define the  class
\[\mathcal{ST}_n[A,B]:= \left\{f\in\mathcal{A}_n: \frac{zf'(z)}{f(z)} \in \mathcal{P}_n[A,B]\right\}.\] This is the well-known class of Janowski starlike  functions. For this class, we have the following
results.

\begin{theorem}Let $-1<B < A\leq 1$ and either (i) $ 1+A \leq \sqrt{2}(1+B) $  and $  2\sqrt{2} (1-B^2)\leq 3 (1-AB) <3 \sqrt{2} (1-B^2)$,
or (ii)  $ (A-B)(1-B^2)+(1-B^2)^2\leq  (1-B^2)\sqrt{(1-B^2)-(1-AB)^2 } +(1-AB)^2 $ and $ 2\sqrt{2} (1-B^2)\geq 3 (1-AB)$. Then $\mathcal{ST}_n[A,B]\subset  \mathcal{SL}_n$.
\end{theorem}

\begin{proof} Since $\frac{zf'(z)}{f(z)}\in P_n[A,B]$, Lemma~\ref{lem3} gives
\begin{equation}\label{th24}
\left|\frac{zf'(z)}{f(z)}-\frac{1-AB}{1-B^2}\right|\leq\frac{A-B}{1-B^2}\quad (|z|<1).
\end{equation}
Let $a=(1-AB)/(1-B^2)$, and suppose the two conditions in (i) hold.  By multiplying the inequality $ 1+A \leq \sqrt{2}(1+B) $ by the positive
constant $1-B$ and rewriting, it is seen that  the given inequality is equivalent to $A-B \leq \sqrt{2}(1-B^2)- (1-AB)$. A division by $1-B^2$
shows that the condition $ 1+A \leq \sqrt{2}(1+B) $ is equivalent to the condition  $(A-B)/(1-B^2) \leq \sqrt{2}-a$. Similarly, the condition
 $  2\sqrt{2} (1-B^2)\leq 3 (1-AB) <3 \sqrt{2} (1-B^2)$ is equivalent to $2 \sqrt{2}/3 \leq a <\sqrt{2}$. In view of these equivalences,
 it follows from \eqref{th24} that the quantity $w=zf'(z)/f(z)$ lies in the disk  $|w-a| <r_a$ where $r_a= \sqrt{2}-a$.  Since
 $2 \sqrt{2}/3 \leq a <\sqrt{2}$ and $|w-a| <r_a$, Lemma~\ref{lem25} shows that  $|w^2-1|<1$ or
  \[  \left|\left(\frac{zf'(z)}{f(z)}\right)^2-1\right|<1. \] This proves that $f\in \mathcal{SL}_n$. The proof is similar if the conditions
  in (ii) hold, and is therefore omitted.
\end{proof}

\begin{theorem}\label{th2.5} Let $-1\leq  B < A\leq 1$, with $B\leq 0$.    Then the $\mathcal{SL}_n$-radius for the class $\mathcal{ST}_n[A,B]$ is
 \[
R_{\mathcal{SL}_n} \left(\mathcal{ST}_n[A,B]\right) = \min\left(1, \left(\frac{2(\sqrt2-1)}{(A-B)+\sqrt{(A-B)^2+4(\sqrt2B-A)B(\sqrt2-1)}}\right)^{\frac{1}{n}}\right) .
\] In particular, if   $ 1+A < \sqrt{2}(1+B) $, then $ \mathcal{ST}_n[A,B] \subseteq \mathcal{SL}_n$. Also the $ \mathcal{SL}$-radius for the class
consisting of starlike functions is $3-2\sqrt{2}$.
\end{theorem}

\begin{proof} Since $\frac{zf'(z)}{f(z)}\in P_n[A,B]$, Lemma~\ref{lem3} yields \[\left|\frac{zf'(z)}{f(z)}-\frac{1-ABr^{2n}}{1-B^2r^{2n}}\right|\leq\frac{(A-B)r^n}{1-B^2r^{2n}}.\]
Since $B\leq 0$, it follows that \[ a:=\frac{1-ABr^{2n}}{1-B^2r^{2n}}\geq 1.\] Using Lemma \ref{lem25}, the function $f$ satisfies
 \[ \left|\left(\frac{zf'(z)}{f(z)}\right)^2-1\right|<1  \] provided
\[ \frac{(A-B)r^n}{1-B^2r^{2n}}<\sqrt2- \frac{1-ABr^{2n}}{1-B^2r^{2n}},\]
that is, \[ (\sqrt2B-A)Br^{2n}+(A-B)r^n-(\sqrt2-1)<0. \] Solving the inequality, we get  $r \leq R_{\mathcal{SL}_n}
\left(\mathcal{ST}_n[A,B]\right)$. The result is sharp for the function given by $f(z)= z(1+Bz^n)^\frac{A-B}{nB}$ for $B\neq 0$ and $f(z)=
z\exp(Az^n/n)$ for $B=0$. Such function $f$ satisfies the equation $zf'(z)/f(z)=(1+Az^n)/(1+Bz^n)$, and therefore the function $f\in \mathcal{ST}_n[A,B]$.
\end{proof}

\begin{theorem}Assume that $f\in \mathcal{ST}_n[A,B]$ and $0< B <A \leq 1$.  Let  $R_1$ be given by
\[R_1= \left(\frac{2\sqrt{2}-3}{(2\sqrt{2}B-3A)B }\right)^{1/(2n)}, \] and let $R_2$ be the number $ R_{\mathcal{SL}_n} \left(\mathcal{ST}_n[A,B]\right) $ as
given in Theorem~\ref{th2.5}.   Let $R_3$ be the largest number in $(0,1]$ such that
\[(A-B)r^n(1-B^2r^{2n})+(1-B^2r^{2n})^2-(1-ABr^{2n})^2-\sqrt{(1-B^2r^{2n})^2-(1-ABr^{2n})^2}\leq 0    \]
for all   $ 0\leq r \leq R_3$.  Then the  $ \mathcal{SL}_n$-radius for the class $\mathcal{ST}_n[A,B]$ is given by
 \[
R_{\mathcal{SL}_n} \left(\mathcal{ST}_n[A,B]\right) = \begin{cases}  R_2 & (R_2\leq R_1) \\ R_3 & (R_2 > R_1).\end{cases}\]
\end{theorem}

\begin{proof}From the proof of the previous theorem,  it easy to see that the quantity $w=zf'(z)/f(z)$ lies in the disk
$|w-a|\leq R$ where  \[ a:=\frac{1-ABr^{2n}}{1-B^2r^{2n}},\quad R= \frac{(A-B)r^n}{1-B^2r^{2n}}.\] Let us first assume that
$R_2 \leq R_1$ where $R_1, R_2$ are as defined in the statement of the theorem. In this case, $r \leq R_1$ if and only if  $a\geq 2\sqrt{2}/3$ and in particular,
for  $0 \leq r \leq R_2$, we have $a\geq 2\sqrt{2}/3$. Lemma \ref{lem25} shows that $f\in \mathcal{SL}_n$ in $|z|\leq r$ if
$R\leq  \sqrt{2}-a $ or equivalently if $r\leq R_2$.

Let us now assume that $R_2>R_1$. In this case, $r \geq R_1$ if and only if  $a\leq 2\sqrt{2}/3$ and in particular for
$r \geq R_2$, we have $a \leq 2 \sqrt{2}/{3}$. Lemma \ref{lem25} shows that $f\in \mathcal{SL}_n$ in $|z|\leq r$ if
$R\leq  \left(\sqrt{1-a^2}-(1-a^2)\right)^{1/2}$ or equivalently if $r\leq R_3$. The sharpness follows because $w=zf'(z)/f(z) $
with $z\in \mathbb{D}$
 fills the entire disk
$|w-a|<R$ where $a$ and $R$ are as given above.
\end{proof}

\section{The $\mathcal{M}_n(\beta)$-Radius Problems}

In this section, we compute the  $\mathcal{M}_n(\beta)$-radii for the classes  $\mathcal{S}_n$  and $\mathcal{CS}_n(\alpha)$.

\begin{theorem}\label{2th2}   The $\mathcal{M}_n(\beta)$-radius of functions in $\mathcal{S}_n$ is given by
\[ R_{\mathcal{M}_n(\beta)}(\mathcal{S}_n) = \left[ \frac{\beta-1}{n+\sqrt{n^2+(\beta-1)^2}}\right]^{1/n}.\]
\end{theorem}

\begin{proof} Since $ h(z) = f(z)/z  \in \mathcal{ P}_n$, Lemma~\ref{lem2} yields
\[ \left|\frac{zf'(z)}{f(z)} -1\right|= \left|\frac{zh'(z)}{h(z)}\right|
\leq \frac{2nr^{n}}{1-r^{2n}}. \] Therefore
\[ \RE \frac{zf'(z)}{f(z)} \leq  \frac{1+2nr^n-r^{2n}}{1-r^{2n}} \leq  \beta  \]
for $r \leq R_{\mathcal{M}_n(\beta)}(\mathcal{S}_n)  $.

     The result is sharp for the function
\[ f(z)  =  \frac{z(1  +  z^{n})}{1  -  z^{n}} \]
which  satisfies  the hypothesis   of Theorem~\ref{2th2}.
\end{proof}

For the class $\mathcal{CS}_n(\alpha)$, the following radius is obtained.

\begin{theorem}\label{2th1}  The $\mathcal{M}_n(\beta)$-radius of functions in $\mathcal{CS}_n(\alpha)$ is given by
\[ R_{\mathcal{M}_n(\beta)}(\mathcal{CS}_n(\alpha)) =\frac{\beta-1}{(1+n-\alpha)+\sqrt{(1+n-\alpha)^{2}+(\beta-1)(1+\beta-2\alpha)}}.\]
\end{theorem}

\begin{proof} Define the function $h$ by
\[ h(z) := \frac{f(z)}{g(z)}. \]
Then $h \in\mathcal{P}_n$ and  by    Lemma~\ref{lem2},
\begin{equation}\label{2th1e1} \left|\frac{zh'(z)}{h(z)}\right| \leq \frac{2nr^n}{1-r^{2n}} .\end{equation}
 Since $g\in \mathcal{ST}_n(\alpha)$, it follows that
  ${zg'(z)}/{g(z)}$ is in $\mathcal{P}_n(\alpha)$
and  therefore, by    Lemma~\ref{lem3},
\begin{equation}\label{2th1e2} \left|\frac{zg'(z)}{g(z)} - \frac{1+(1-2\alpha) r^{2n}}{1-r^{2n}}\right|
 \leq \frac{2(1-\alpha)r^n}{1-r^{2n}}. \end{equation}
Since
\[ \frac{zf'(z)}{f(z)} =  \frac{zg'(z)}{g(z)}+\frac{zh'(z)}{h(z)}, \]
in view of (\ref{2th1e1}) and (\ref{2th1e2}),  it is seen that
\begin{equation*}\label{2eq6}
 \left|\frac{zf'(z)}{f(z)} -
\frac{1+(1-2\alpha)r^{2n}}{1-r^{2n}}\right| \leq \frac{2(1+n-\alpha)r^n}{1-r^{2n}}. \end{equation*}

This represents a circular disk intersecting the real axis at
$$ x_0 = \frac{1-2(1+n-\alpha)r^n+(1-2\alpha)r^{2n}}{1-r^{2n}}
\mbox{ and }  x_1 = \frac{1+2(1+n-\alpha)r^n+(1-2\alpha)r^{2n}}{1-r^{2n}}, $$ and therefore
\[ \RE \frac{zf'(z)}{f(z)} \leq  \frac{1+2(1+n-\alpha)r^n+(1-2\alpha)r^{2n}}{1-r^{2n}} \leq  \beta  \]
for $r \leq R$.

The  function
\begin{equation*} \label{fn1}
 f(z) = \frac{z(1 + z^n)}{(1 -z^n)^{(n+2-2\alpha)/n}} \end{equation*}
satisfies  the  hypothesis  of  Theorem~\ref{2th1} with \[ g(z) =\frac{z}{(1-z^n)^{(2-2\alpha)/n}}.\] Since
\[ \frac{zf'(z)}{f(z)} = \frac{1+2(1+n-\alpha)z^n+(1-2\alpha)z^{2n}}{1-z^{2n}}=\beta \]
for $z=R= R_{\mathcal{M}_n(\beta)}(\mathcal{CS}_n(\alpha)) $, the result is sharp.
\end{proof}

\end{document}